\documentclass[11pt,a4paper,fleqn]{article}

\usepackage{amsmath, amssymb, amsthm}
\usepackage[top=2.5cm, bottom=2.5cm, left=1.7cm, right=1.7cm]{geometry}
\usepackage[colorlinks=true,urlcolor=blue,pdfstartview=FitH, linkcolor=blue,citecolor=blue,bookmarksopen, bookmarksnumbered={true}]{hyperref}
\usepackage{enumerate}
\usepackage{cancel}
\usepackage{mathpazo}

\usepackage{pdfsync}
\usepackage{color}
\newtheorem{theorem}{Theorem}[section]

\newtheorem{corollary}{Corollary}[section]

\newtheorem{problem}{Problem}[section]

\definecolor{pink}{rgb}{1,0.08,0.58}
\definecolor{orange}{rgb}{1,0.5,0}
\definecolor{purple}{rgb}{0.75002,0,1}
\definecolor{olive}{RGB}{85,107,47}
\definecolor{mygreen}{rgb}{0,0.6,0}

\def\hh{{\mathcal{H}}}
\def\RR{{\mathbb{R}}}
\def\NN{{\mathbb{N}}}
\def\Z{{\mathcal{Z}}}
\def\B{{\mathcal{B}}}
\def\M{{\mathcal{M}}}
\def\N{{\mathcal{N}}}
\def\T{{\mathcal{T}}}
\date{}

\title{Some applications of Caristi's fixed point theorem in metric spaces}

\author{Farshid Khojasteh\footnote{Department of Mathematics, Arak-Branch, Islamic Azad University, Arak, Iran,\;\href{mailto:fr\_khojasteh@yahoo.com}{fr\_khojasteh@yahoo.com}.
                         }
\; , \; Erdal Karapinar\footnote{Department of Mathematics, Atilim University, $\dot{I}$ncek, 06836 Ankara, Turkey,\;\href{mailto:erdalkarapinar@yahoo.com}{erdalkarapinar@yahoo.com}.
}
\; , \; Hassan Khandani\footnote{Department of Mathematics, Mahabad-Branch, Azad University, Mahabad, Iran,\;\href{mailto:khandani.hassan@yahoo.com}{khandani.hassan@yahoo.com}.
}
}

\begin{document}
\maketitle
\date{}
\hrule

\abstract{
In this work, partial answers to Reich, Mizoguchi and Takahashi, and Amini-Harandi's conjectures are presented via a light version of Caristi's fixed point theorem. Moreover, we introduce that many of known fixed point theorem can easily derived from the Caristi's theorem.
Finally, existence of bounded solutions of a functional equation is studied.
}
\hrule

\section{Introduction and preliminaries}
In the literature, the Caristi fixed-point theorem  is known as  one of the very interesting and useful generalization of the Banach fixed point theorem for self-mappings on a complete metric space. In fact,  Caristi fixed-point theorem is a modification of the $\varepsilon$-variational principle of Ekeland
(\cite{E1974, E1979}) that is a crucial tool in the nonlinear analysis, in particular, optimization, variational inequalities, differential equations
and control theory. Furthermore, in 1977 Western  \cite{W1977}  proved that the conclusion of Caristi's theorem is equivalent to metric completeness.
In the last decades, Caristi's fixed-point theorem has been generalized and extended in several directions (see e.g., \cite{1,2}
and the related references therein).

{\par The Caristi's fixed point theorem asserted as follows:
\begin{theorem}\cite{CAR}\label{t1}
Let $(X,d)$ be a complete metric space and let $T:X\to X$ be a mapping such that
\begin{equation}\label{eq1}
d(x,Tx)\leq \varphi(x)-\varphi(Tx)
\end{equation}
for all $x\in X$, where $\varphi:X\to[0,+\infty)$ be a lower semi continuous mapping. Then $T$ has at least a fixed point.
\end{theorem} }
Let us recall some basic notations, definitions and well-known results
needed in this paper. Throughout this paper, we denote by $\NN$ and $\RR$, the sets of positive integers and real numbers, respectively. Let $(X,d)$
be a metric space. Denote by $\mathcal{CB}(X)$ the family of all nonempty closed and
bounded subsets of $X$. A function $\mathcal{H}:\mathcal{CB}(X)\times
\mathcal{CB}(X)\rightarrow \lbrack 0,\infty )$ defined by
\[
\mathcal{H}(A,B)=\textit{max}\left\{ \sup_{x\in B}d(x,A)\textit{,}\sup_{x\in
A}d(x,B)\right\}
\]
is said to be the Hausdorff metric on $\mathcal{CB}(X)$ induced by the
metric $d$ on $X$. A point $v$ in $X$ is a fixed point of a map $T$ if $v=Tv$
(when $T:X\rightarrow X$ is a single-valued map) or $v\in Tv$ (when $%
T:X\rightarrow \mathcal{CB}(X)$ is a multi-valued map).

Let $(X,d)$ be a complete metric space and $T :X\to X$ a map. Suppose there
exists a function $\phi: [0,+\infty) \to [0,+\infty)$ satisfying $\phi(0) = 0$, $\phi(s) < s$ for $s > 0$ and that
$\phi$ is right upper semi-continuous such that
\[
d(Tx,Ty)\leq\phi(d(x,y)) \ \ \ x,y\in X.
\]
Boyd-Wong \cite{Boyd} showed that $T$ has a unique fixed point.

In 1972, Reich \cite{reich1} introduced the following open problem:
\begin{problem}\label{pr1}
Let $(X,d)$ be a complete metric space and let $T:X\to\mathcal{CB}(X)$ be a multi-valued mapping such that
\begin{equation}
\hh(Tx,Ty)\leq \mu(d(x,y))
\end{equation}
for all $x,y \in X$, where $\mu :\RR^+\rightarrow \RR^+$ is continuous and increasing map such that $\mu(t)<t$, for all $t>0$. Does $T$ have a fixed point?
\end{problem}
Some partial answers to Problem \ref{pr1} given by Daffer et al.(1996) \cite{daffer} and Jachymski(1998) \cite{jachymski}. In these works, the authors
consider additional conditions on the mapping $\mu$ to find a fixed point.
\begin{itemize}
\item[$\bullet$] Daffer et. al
assumed that
$\mu :\RR^+\rightarrow \RR^+$
\begin{itemize}
\item is upper right semi continuous,
\item $\mu(t)<t$ for all $t>0$ and,
\item $\mu(t)\leq t-at^b$, where $a>0$, $1<b<2$ on some interval $[0,s]$, $s>0$.
\end{itemize}
\item[$\bullet$] and Jachymski
assume that
$\mu :\RR^+\rightarrow \RR^+$
\begin{itemize}
\item is supperadditive, i.e., $\mu(x+y)>\mu(x)+\mu(y)$, for all $x,y\in \RR^+$ and,
\item $t\mapsto t-\mu(t)$ is nondecreasing.
\end{itemize}
\end{itemize}
In 1983, Reich \cite{reich2}, introduced another problem as follows:
\begin{problem}\label{pr2}
Let $(X,d)$ be a complete metric space and let $T:X\to \mathcal{CB}(X)$ be a mapping such that
\begin{equation}\label{e4y5}
\mathcal{H}(Tx,Ty)\leq \eta(d(x,y))d(x,y)
\end{equation}
for all $x,y \in X$, where $\eta:(0,+\infty)\to [0,1)$ be a mapping such that $\limsup_{r\to t^+}\eta(r)<1$, for all $r\in(0,+\infty)$. Does $T$ have a fixed point?
\end{problem}
In 1989, Mizoguchi and Takahashi\cite{Miz}, gave a partial answer to Problem \ref{pr2} as follows:
\begin{theorem}
Let $(X,d)$ be a complete metric space and let $T:X\to \mathcal{CB}(X)$ be a mapping such that
\begin{equation}\label{e45}
\mathcal{H}(Tx,Ty)\leq \eta(d(x,y))d(x,y)
\end{equation}
for all $x,y \in X$, where $\eta:(0,+\infty)\to [0,1)$ be a mapping such that $\limsup_{r\to t^+}\eta(r)<1$, for all $r\in[0,+\infty)$. Then $T$ has a fixed point.
\end{theorem}
Another analogous open problem, raised in 2010 by Amini-Harandi\cite{Amini} which we assert it after the following notations:

In what follows, $\gamma:[0,+\infty)\to [0,+\infty)$ be subadditive, i.e.$\gamma(x+y)\leq \gamma(x)+\gamma(y)$, for each $x,y\in[0,+\infty)$, a nondecreasing continuous
map such that $\gamma^{-1}(\{0\})=\{0\}$, and let $\Gamma$ consist of all such functions. Also, let $\mathcal{A}$ be the class of all maps $\theta:[0,+\infty)\to [0,+\infty)$ for which there exists an $\epsilon_0 > 0$ such that
\[
\theta(t)\leq \epsilon_0 \ \ \Rightarrow \ \ \theta(t)\geq\gamma(t)
\]
where $\gamma\in\Gamma$.
\begin{problem}\label{pr3}
Assume that $T:X\to \mathcal{CB}(X)$ is a weakly contractive set-valued map on a complete metric space $(X,d)$, i.e.,
\[
\hh(Tx,Ty)\leq d(x,y)-\theta(d(x,y))
\]
for all $x,y\in X$, where $\theta\in\mathcal{A}$. Does $T$ have a fixed point?
\end{problem}
The answer is yes if $Tx$ is compact for every $x$ (Amini-Harandi \cite[Theorem 3.3]{Amini}).\\~\\
In this work, we show that many of known Banach contraction's generalization can be deduce and generalize by Caristi's fixed point theorem and its consequences. Also, partial answers to mentioned open problems are given via our main results. For more details about fixed point generalization of multi-valued mappings we refer to \cite{Tok}.
\section{Main Result}
In this section, we show that many of known fixed point results can be deduces from the following light version of Caristi's theorem:

\begin{corollary}\label{co1}
Let $(X,d)$ be a complete metric space, and let $T:X\to X$ be a mapping such that
\begin{equation}
d(x,y)\leq \varphi(x,y)-\varphi(Tx,Ty),
\end{equation}
for all $x,y\in X$, where $\varphi:X\times X\to [0,\infty)$ is a lower semi continuous with respect to first variable. Then $T$ has a unique fixed point.
\end{corollary}
\begin{proof}
For each $x\in X$, let $y=Tx$ and $\psi(x)=\varphi(x,Tx)$. Then for each $x\in X$
\[
d(x,Tx)\leq \psi(x)-\psi(Tx)
\]
and $\psi$ is a lower semi continuous mapping. Thus, applying Theorem \ref{t1} conclude desired result. To see the uniqueness of fixed point suppose that
$u,v$ be two distinct fixed point for $T$. Then
\[
d(u,v)\leq \varphi(u,v)-\varphi(Tu,Tv)=\varphi(u,v)-\varphi(u,v)=0.
\]
Thus, $u=v$.
\end{proof}
\begin{corollary}\cite[Banach contraction principle]{BAN}
Let $(X,d)$ be a complete metric space and let $T:X\to X$ be a mapping such that for some $\alpha\in[0,1)$
\begin{equation}\label{e6}
d(Tx,Ty)\leq \alpha d(x,y)
\end{equation}
for all $x,y\in X$. Then $T$ has a unique fixed point.
\end{corollary}
\begin{proof}
Define, $\varphi(x,y)=\frac{d(x,y)}{1-\alpha}$. Then (\ref{e6}) shows that
\begin{equation}
(1-\alpha)d(x,y)\leq d(x,y)-d(Tx,Ty).
\end{equation}
It means that
\begin{equation}
d(x,y)\leq \frac{d(x,y)}{1-\alpha}-\frac{d(Tx,Ty)}{1-\alpha}
\end{equation}
and so
\begin{equation}
d(x,y)\leq \varphi(x,y)-\varphi(Tx,Ty)
\end{equation}
and so by applying Corollary \ref{co1}, one can conclude that $T$ has a unique fixed point.
\end{proof}
\begin{corollary}
Let $(X,d)$ be a complete metric space and let $T:X\to X$ be a mapping such that
\begin{equation}\label{e8}
d(Tx,Ty)\leq \eta(d(x,y))
\end{equation}
where $\eta:[0,+\infty)\to[0,\infty)$ be a lower semi continuous mapping such that $\eta(t)<t$, for each $t>0$ and $\frac{\eta(t)}{t}$ be a non-decreasing map. Then $T$ has a unique fixed point.
\end{corollary}
\begin{proof}
Define, $\varphi(x,y)=\frac{d(x,y)}{1-\frac{\eta(d(x,y))}{d(x,y)}}$, if $x\neq y$ and otherwise $\varphi(x,x)=0$. Then (\ref{e8}) shows that
\begin{equation}
(1-\frac{\eta(d(x,y))}{d(x,y)})d(x,y)\leq d(x,y)-d(Tx,Ty).
\end{equation}
It means that
\begin{equation}
d(x,y)\leq \frac{d(x,y)}{1-\frac{\eta(d(x,y))}{d(x,y)}}-\frac{d(Tx,Ty)}{1-\frac{\eta(d(x,y))}{d(x,y)}}.
\end{equation}
Since $\frac{\eta(t)}{t}$ is non-decreasing and $d(Tx,Ty) < d(x,y)$ thus
\begin{equation}
d(x,y)\leq \frac{d(x,y)}{1-\frac{\eta(d(x,y))}{d(x,y)}}-\frac{d(Tx,Ty)}{1-\frac{\eta(d(Tx,Ty))}{d(Tx,Ty)}}=\varphi(x,y)-\varphi(Tx,Ty)
\end{equation}
and so by applying Corollary \ref{co1}, one can conclude that $T$ has a unique fixed point.
\end{proof}

The following results are the main result of this paper and play the crucial role to find the partial answers for Problem \ref{pr1}, Problem \ref{pr2} and Problem \ref{pr3}. Comparing the partial answers for Reich's problems, our answers included simple conditions. Also, the compactness condition on $Tx$ is not needed.
\begin{theorem}\label{ttt4}
Let $(X,d)$ be a complete metric space, and let $T:X\to \mathcal{CB}(X)$ be a non-expansive mapping such that
for each $x\in X$ and for all $y\in Tx$, there exists $z\in Ty$ such that
\begin{equation}\label{fdre}
d(x,y)\leq \varphi(x,y)-\varphi(y,z),
\end{equation}
where $\varphi:X\times X\to [0,\infty)$ is a lower semi continuous with respect to first variable. Then $T$ has a  fixed point.
\end{theorem}
\begin{proof}
Let $x_0\in X$ and let $x_{1}\in Tx_0$. If $x_0=x_1$ then $x_0$ is a fixed point and finished. Otherwise, let $x_1\neq x_0$. By assumption there exists
$x_2\in Tx_1$ such that
\[
d(x_0,x_1)\leq \varphi(x_0,x_1)-\varphi(x_1,x_2).
\]
Alternatively, one can choose $x_{n}\in Tx_{n-1}$ such that $x_n\neq x_{n-1}$ and find $x_{n+1}\in Tx_n$ such that
\begin{equation}\label{g32}
0<d(x_{n-1},x_n)\leq \varphi(x_{n-1},x_n)-\varphi(x_n,x_{n+1})
\end{equation}
which means that, $\{\varphi(x_{n-1},x_n)\}_n$ is a non-increasing and bounded below sequence so it is converges to some $r\geq 0$. By taking limit on both side of (\ref{g32}) we have $\lim\limits_{n\to\infty}d(x_{n-1},x_n)=0$. Also, for all $m,n\in\NN$ with $m>n$,
\begin{equation}\label{eh3}
\begin{array}{lll}
d(x_n,x_m)&\leq &\overset{m}{\underset{i=n+1}{\sum}}d(x_{i-1},x_{i})\\\\
&\leq &\overset{m}{\underset{i=n+1}{\sum}}\varphi(x_{i-1},x_i)-\varphi(x_i,x_{i+1})\\\\
&\leq &\varphi(x_{n},x_{n+1})-\varphi(x_m,x_{m+1}).
\end{array}
\end{equation}
Therefore, by taking limsup on both side of (\ref{eh3}) we have
\[
\lim_{n\to\infty}(\sup\{d(x_n,x_m):m>n\})=0.
\]
It means that, $\{x_n\}$ is a Cauchy sequence and so it is converges to $u\in X$. Now we show that $u$ is a fixed point of $T$.
\begin{equation}\label{e2k2}
\begin{array}{lll}
d(u,Tu)&\leq &d(u,x_{n+1})+d(x_{n+1},Tu)\\\\
&= &d(u,x_{n+1})+\hh(Tx_{n},Tu)\\\\
&\leq & d(u,x_{n+1})+d(x_{n},u).
\end{array}
\end{equation}
By taking limit on both side of (\ref{e2k2}), we get $d(x,Tx)=0$ and this means that $x\in Tx$.
\end{proof}
The following theorem is a partial answer to Problem \ref{pr1}
\begin{theorem}\label{tt1} Let $(X,d)$ be a complete metric space, and let $T:X\to \mathcal{CB}(X)$ be a multi-valued function such that
\[
\hh(Tx,Ty)\leq \eta (d(x,y))
\]
for all $x,y \in X$, where $\eta :[0,\infty)\rightarrow [0,\infty)$ is lower semi continuous map such that $\eta(t)<t$, for all $t\in (0,+\infty)$ and $\frac {\eta(t)}{t}$ is non-decreasing.
Then $T$ has a fixed point.
\end{theorem}
\begin{proof}Let $x\in X$ and $y\in Tx$. If $y=x$  then $T$ has a fixed point and the proof is complete, so we suppose that $y\neq x$. Define
\[\theta (t)=\frac{\eta(t)+t}{2}\hbox{   for all $t\in (0,+\infty)$}.\]
Since $\hh(Tx,Ty)\leq \eta (d(x,y))<\theta (d(x,y))<d(x,y)$. Thus there exists $\epsilon_0>0$ such that $\theta (d(x,y))=\hh(Tx,Ty)+\epsilon_0$. So there exists $z\in Ty$ such that
\begin{equation}\label{rd2}
d(y,z)<\hh(Tx,Ty)+\epsilon_0=\theta (d(x,y))<d(x,y).
\end{equation}
We again suppose that $y\neq z$, therefore $d(x,y)-\theta (d(x,y))\leq d(x,y)-d(y,z)$ or equivalently
\[d(x,y)<\frac{d(x,y)}{1-\frac{\theta (d(x,y))}{d(x,y)}}-\frac{d(y,z)}{1-\frac{\theta (d(x,y))}{d(x,y)}}\]
since $\frac{\theta(t)}{t}$ is also a nondecreasing function and $d(y,z)<d(x,y)$  we get
\[d(x,y)<\frac{d(x,y)}{1-\frac{\theta (d(x,y))}{d(x,y)}}-\frac{d(y,z)}{1-\frac{\theta (d(y,z))}{d(y,z)}}.\]
Define $\Phi(x,y)=\frac{d(x,y)}{1-\frac{\theta (d(x,y))}{d(x,y)}}$ if $x\neq y$, otherwise $0$ for all $x,y\in X$. It means that,
\[
d(x,y)<\Phi(x,y)-\Phi(y,z).
\]
Therefore, $T$ satisfies in (\ref{fdre}) of Theorem \ref{ttt4} and so conclude that $T$ has a unique fixed point $u$ and the proof is completed.
\end{proof}


The following theorem is a partial answer to Problem \ref{pr2}:
\begin{corollary}\cite[Mizoguchi-Takahashi's type]{Miz}\label{co22}
Let $(X,d)$ be a complete metric space and let $T:X\to \mathcal{CB}(X)$ be a multi-valued mapping such that
\begin{equation}\label{e7}
\hh(Tx,Ty)\leq \eta(d(x,y))d(x,y)
\end{equation}
for all $x,y\in X$, where $\eta:[0,+\infty)\to[0,1)$ be a lower semi continuous and non-decreasing mapping. Then $T$ has a fixed point.
\end{corollary}
\begin{proof}
Let $\theta (t)=\eta(t)t$ . $\theta(t)<t$ for all $t\in R_{+}$  and $\frac{\theta(t)}{t}=\eta (t)$  is a nondecreasing mapping. By the assumption $ d(Tx,Ty)\leq \eta (d(x,y))d(x,y)=\theta(d(x,y))$ for all $x,y\in X$, therefore by Theorem \ref{tt1} $T$ has a fixed point.
\end{proof}
Note that if $\eta :[0,\infty)\rightarrow [0,1)$ be an non-decreasing map then for all $s\in[0,+\infty)$
\[
\begin{array}{lll}
\limsup\limits_{t\to s^+}\eta(t)&=&\inf\limits_{\delta>0}\sup\limits_{s\leq t<s+\delta}\eta(t)\\
&=&\lim\limits_{\delta\to 0}\sup\limits_{s\leq t<s+\delta}\eta(t)\leq\eta(s+\delta)<1.
\end{array}
\]
It means that, Corollary \ref{co22} yields from the Mizoguchi-Takahashi's results directly and here we deduce it from our results\cite{Miz}.

The following theorem is a partial answer to Problem \ref{pr3}:
\begin{corollary}\label{tt2} Let $(X,d)$ be a complete metric space, and let $T:X\to \mathcal{CB}(X)$ be a multi-valued function such that
\[
\hh(Tx,Ty)\leq d(x,y)-\theta(d(x,y))
\]
for all $x,y \in X$, where $\theta :(0,\infty)\rightarrow (0,\infty)$ is upper semi continuous map such that, for all $t\in (0,+\infty)$ and $\frac {\theta(t)}{t}$ is non-increasing.
Then $T$ has a fixed point.
\end{corollary}
\begin{proof}
Let $\eta(t)=t-\theta(t)$, for each $t>0$. Then, $\eta(t)<t$, for each $t>0$ and $\frac{\eta(t)}{t}=1-\frac{\theta(t)}{t}$ is non-decreasing. Thus, desired result is obtained by Theorem \ref{tt1}.
\end{proof}

\section{Existence of bounded solutions of functional equations}
Mathematical optimization is one of the fields in which the methods of fixed point theory are widely
used. It is well known that the dynamic programming provides useful tools for mathematical optimization
and computer programming. In this setting, the problem of dynamic programming related to multistage
process reduces to solving the functional equation
\begin{equation}\label{sw9}
p(x)=\sup_{y\in \T}\{f(x,y)+\Im(x,y,p(\eta(x,y)))\}, \ \ \ x\in \Z,
\end{equation}
where $\eta:\Z\times \T\to \Z$, $f:\Z\times \T\to \RR$ and $\Im:\Z\times \T\times \RR\to \RR$.
We assume that $\M$ and $\N$ are Banach spaces,
$\Z\subset \M$ is a state space and $\T\subset \N$ is a decision space. The studied process consists
of \textit{a state space}, which is the set of the initial state, actions and transition model of the process and \textit{a decision space}, which is the set of possible actions that are allowed for the process.

Here, we study the existence of the bounded solution of the functional equation \ref{sw9}.
Let $\B(\Z)$ denote the set of all bounded real-valued functions on $W$ and, for an arbitrary $h\in \B(\Z)$, define $||h|| = \sup_{x\in \Z} |h(x)|$. Clearly, $(\B(W),||.||)$ endowed with the metric d defined by
\begin{equation}\label{p2}
d(h,k)=\sup_{x\in\Z}|h(x)-k(x)|,
\end{equation}
for all $h,k\in \B(\Z)$, is a Banach space. Indeed, the convergence in the space $\B(\Z)$ with respect to $||.||$ is uniform. Thus, if we consider a Cauchy sequence $\{h_n\}$ in $\B(\Z)$, then $\{h_n\}$ converges uniformly to a function, say $h^*$, that is bounded and so $h\in \B(\Z)$.\\
We also define $S:\B(\Z)\to \B(\Z)$ by
\begin{equation}\label{re2}
S(h)(x)=\sup_{y\in \T}\{f(x,y)+\Im(x,y,h(\eta(x,y)))\}
\end{equation}
for all $h\in \B(\Z)$ and $x\in \Z$.

We will prove the following theorem.
\begin{theorem}
Let $S : \B(\Z)\to \B(\Z)$ be an upper semi-continuous operator defined by (\ref{re2}) and assume that the
following conditions are satisfied:
\begin{itemize}
\item[$(i)$] $f:\Z\times \T\to \RR$ and $\Im:\Z\times \T\times \RR\to \RR$ are continuous and bounded;
\item[$(ii)$] for all $h,k\in \B(\Z)$, if
\begin{equation}\label{pp2}
\begin{array}{rll}
0<d(h,k)<1 &  implies&  |\Im(x,y,h(x))-\Im(x,y,k(x))|\leq \frac{1}{2}d^2(h,k)\\
d(h,k)\geq 1 & implies & |\Im(x,y,h(x))-\Im(x,y,k(x))|\leq \frac{2}{3}d(h,k)
\end{array}
\end{equation}
where $x\in \Z$ and $y\in \T$.
Then the functional equation (\ref{sw9}) has a bounded solution.
\end{itemize}
\end{theorem}
\begin{proof}
Note that $(\B(\Z),d)$ is a complete metric space, where $d$ is the metric given by (\ref{p2}). Let $\mu$ be an
arbitrary positive number, $x\in \Z$ and $h_1, h_2 \in \B(\Z)$, then there exist $y_1, y_2 \in \T$ such that
\begin{eqnarray}
S(h_1)(x)&<&f(x,y_1)+\Im(x,y_1,h_1(\eta(x,y_1)))+\mu,\\
S(h_2)(x)&<&f(x,y_2)+\Im(x,y_2,h_2(\eta(x,y_2)))+\mu,\\
S(h_1)(x)&\geq& f(x,y_1)+\Im(x,y_1,h_1(\eta(x,y_1))),\\
S(h_2)(x)&\geq&f(x,y_2)+\Im(x,y_2,h_2(\eta(x,y_2))).
\end{eqnarray}
Let $\varrho :[0,\infty)\rightarrow [0,\infty)$ be defined by
\[
\varrho(t)=\left\{
\begin{array}{ll}
\frac {1}{2}t^2, & \hbox{$0<t<1$} \\
\frac{2}{3}t, & \hbox{$t\ge 1$}.
\end{array}
\right.
\]
Then we can say that (\ref{pp2}) is equivalent to
\begin{equation}\label{oo3}
|\Im(x,y,h(x))-\Im(x,y,k(x))|\leq\varrho(d(h,k))
\end{equation}
for all $h,k\in \B(\Z)$. It is easy to see that $\varrho(t)<t$, for all $t>0$ and $\frac{\varrho(t)}{t}$ is a non-decreasing function.

Therefore, by using (24), (27) and (28), it follows that
\[
\begin{array}{lll}
S(h_1)(x)-S(h_2)(x)&<& \Im(x,y_1,h_1(\eta(x,y_1)))-\Im(x,y_2,h_2(\eta(x,y_2)))+\mu\\
&\leq&|\Im(x,y_1,h_1(\eta(x,y_1)))-\Im(x,y_2,h_2(\eta(x,y_2)))|+\mu\\
&\leq&\varrho(d(h_1,h_2))+\mu.
\end{array}
\]
Then we get
\begin{equation}\label{gh3}
S(h_1)(x)-S(h_2)(x)<\varrho(d(h_1,h_2))+\mu.
\end{equation}
Analogously, by using (25) and (26), we have
\begin{equation}\label{gh4}
S(h_2)(x)-S(h_1)(x)<\varrho(d(h_1,h_2))+\mu.
\end{equation}
Hence, from (\ref{gh3}) and (\ref{gh4}) we obtain
\[
|S(h_2)(x)-S(h_1)(x)|<\varrho(d(h_1,h_2))+\mu,
\]
that is,
\[
d(S(h_1),S(h_2))<\varrho(d(h_1,h_2))+\mu.
\]
Since the above inequality does not depend on $x\in\Z$ and $\mu > 0$ is taken arbitrary, then we conclude
immediately that
\[
d(S(h_1),S(h_2))\leq\varrho(d(h_1,h_2)),
\]
so we deduce that the operator $S$ is an $\varrho$-contraction. Thus, due to the continuity of $S$, Theorem \ref{tt1} applies to the operator $S$, which has a fixed point $h^*\in \B(\Z)$, that is, $h^*$ is a bounded solution of the functional equation
(\ref{sw9}).
\end{proof}


\end{document}